\documentclass[11pt]{amsart}

\usepackage{amssymb,amsfonts,amsmath}
\usepackage{amsthm, stmaryrd}
\usepackage{graphicx}
\usepackage[all,arc]{xy}
\usepackage{hyperref}
\hypersetup{colorlinks,allcolors=violet}
\usepackage{enumerate}
\usepackage{bbm}
\usepackage{dsfont}
\usepackage{mathrsfs}
\usepackage{tikz-cd}
\usetikzlibrary{shapes}
\usepackage{import}
\usepackage{float}
\restylefloat{table}
\usepackage{multirow}
\usepackage{pdflscape}
\usepackage{listofitems}
\usepackage{standalone}

\usepackage[margin=1in]{geometry}
\usepackage{mathtools}
\usetikzlibrary{decorations.pathmorphing}
\usepackage{mathrsfs}
\newtheorem{thm}{Theorem}[section]
\newtheorem*{thm*}{Theorem}

\newtheorem{cor}[thm]{Corollary}
\newtheorem{prop}[thm]{Proposition}
\newtheorem{lem}[thm]{Lemma}

\theoremstyle{definition}
\newtheorem{defn}[thm]{Definition}

\newtheorem{exmp}[thm]{Example}

\newtheorem{rem}[thm]{Remark}
\newtheorem*{thm1.2}{\textrm{Theorem 1.2}}

\theoremstyle{remark}
\newcommand{\Mbar}{\overline{\mathcal{M}}}
\newcommand{\Tree}{\mathrm{Tree}}
\newcommand{\M}{\mathcal{M}}

\newtheorem{sch}[thm]{Scholium}

\newcommand{\Z}{\mathbb{Z}}

\newcommand{\QQ}{\mathbb{Q}}

\renewcommand{\P}{\mathbb{P}}
\newcommand{\bbS}{\mathbb{S}}

\newcommand{\Aut}{\operatorname{Aut}}

\newcommand{\val}{\operatorname{val}}

\def\B{\mathbf{B}}
\def\C{\mathbb{C}}

\def\G{\mathbf{G}}

\def\P{\mathbb{P}}
\def\R{\mathbb{R}}
\def\T{\mathsf{T}}

\def\Z{\mathbb{Z}}

\def\calD{\mathcal{D}}

\def\calF{\mathcal{F}}

\def\calW{\mathcal{W}}

\def\M{\mathcal{M}}

\newcommand{\TT}{\mathbf{T}}

\newcommand{\Conf}{\mathrm{Conf}}

\newcommand\cycle[2][\,]{%
  \readlist\thecycle{#2}%
  (\foreachitem\i\in\thecycle{\ifnum\icnt=1\else#1\fi\i})%
}

\makeatletter
\let\c@equation\c@thm
\makeatother
\numberwithin{equation}{section}

\graphicspath{ {images/} }
\title{Automorphisms of the boundary complex of $\Mbar_{0, n}(\P^r, d)$}

\author[A. Joisha]{Arjun Joisha}\address{Lynbrook High School, San Jose CA}\email{\url{ajoisha898@student.fuhsd.org}}

\author[S. Kannan]{Siddarth Kannan}\address{Department of Mathematics, Massachusetts Institute of Technology}
\email{\url{spkannan@mit.edu}}

\begin{document}
\maketitle\thispagestyle{empty}
%\tableofcontents
\begin{abstract}
We compute the automorphism group of the dual complex $\T_{d, n}$ of the boundary divisor in the Kontsevich moduli space $\Mbar_{0, n}(\P^r, d)$. When $d \geq 2$, we find that $\Aut(\T_{d, n}) \cong \bbS_{n}$, while $\Aut(\T_{1, n}) \cong \bbS_{n + 1}$ for all $n \geq 4$. The complex $\T_{1, n}$ is also the dual complex of the boundary divisor in the Fulton--MacPherson compactification of the configuration space of $n$ points on $X$, if $X$ is any smooth, proper, and connected algebraic variety over $\C$. Following work of Massarenti, this implies that $\T_{1, n}$ admits automorphisms which in general do not extend to $X[n]$.
\end{abstract}

\section{Introduction}
The Kontsevich moduli space $\Mbar_{0, n}(\P^r, d)$ parameterizes $n$-pointed stable maps of genus zero to the projective space $\P^r$. The genus-zero Gromov--Witten invariants of $\P^r$ are intersection numbers on this moduli space. It is a smooth Deligne--Mumford stack, and furnishes a normal-crossings compactification of the moduli space $\M_{0, n}(\P^r, d)$ of maps from smooth $n$-pointed rational curves to $\P^r$. Let
\[ \partial \Mbar_{0, n}(\P^r, d) := \Mbar_{0, n}(\P^r, d) \smallsetminus \M_{0, n}(\P^r, d) \]
be the boundary divisor parameterizing maps from nodal curves.
We define $\T_{d, n}$ to be the dual complex of this normal-crossings divisor. There is a natural action of the symmetric group $\bbS_n$ on this dual complex, induced by the action of $\bbS_n$ on the moduli space.
% The space $\T_{d, n}$ can be identified with the following (generalized) simplicial complexes in the geometric topology literature:
% \begin{itemize}
% \item the quotient of the complex of separating curves on a surface $\Sigma_{d, n}$ of genus $d$ with $n$ ordered boundary components by the pure mapping class group action;
% \item the simplicial join of the Robinson-Whitehouse tree space $T_{n + d -1}$ with a standard $d$-simplex, modulo a natural $\bbS_d$-action.
% \end{itemize}
In this note, we use combinatorial methods to compute the automorphism group of $\T_{d, n}$ for all $d \geq 1$ and $n \geq 0$. When $d = 1$, the complex $\T_{1, n}$ is also the dual complex of the boundary divisor in the Fulton--MacPherson compactification \[\Conf_n(X) \subset X[n]\] of the ordered configuration space of $n$ points on any smooth, proper, and connected variety $X$. This is the unique case we consider where the dual complex admits automorphisms which do not extend to the moduli space.

\begin{thm}\label{thm:mainthm}
When $d \geq 2$, we have
\[ \Aut(\T_{d, n}) \cong \bbS_n \]
for all $n$. When $d = 1$, we have
\[ \Aut(\T_{1, n}) \cong \bbS_{n + 1} \]
for $n \geq 4$, while $\Aut(\T_{1, 3}) \cong \bbS_3$ and $\Aut(\T_{1, n})$ is trivial for $n \leq 2$.
\end{thm}
For the $d \geq 2$ case of Theorem \ref{thm:mainthm}, our first step is to show that any automorphism of $\T_{d, n}$ is determined by its action on the $0$-cells. If $\T_{d, n}$ were an ordinary abstract simplicial complex, this property would hold by definition, but care must be taken because $\T_{d, n}$ is generally a \textit{symmetric $\Delta$-complex} in the sense of \cite{cgp}. It is possible for distinct $p$-cells to have the same ordered list of $(p-1)$-dimensional faces, and indeed this occurs in $\T_{d, n}$. On the level of the moduli space, this corresponds to the fact that the dual tree of a stable map is not determined by the dual trees of all possible smoothings of the map. Geometrically, this corresponds to the existence of strata of the divisor which are not determined by their local branches. See Figure \ref{fig:contraction_deck_counterexmp} for an example.

\begin{figure}
    \centering
    \includegraphics{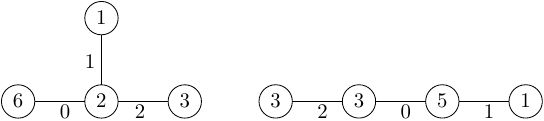}
    \caption{Distinct $2$-cells of $\T_{12, 0}$ which have the same ordered list of $1$-dimensional faces, corresponding to distinct codimension-$3$ strata of the normalization of the boundary of $\Mbar_{0, 0}(\P^r, 12)$ which are not determined by their local branches. The vertex-weighted trees represent dual graphs of stable maps of degree $12$: each vertex corresponds to an irreducible component of the source curve, and the integer weights record the degree of the map on the corresponding component.}
    \label{fig:contraction_deck_counterexmp}
\end{figure}

To get around this issue, we first prove that any element of $\Aut(\T_{d, n})$ must preserve a certain highly symmetric cell, which then implies that the set of degrees of a stable map along irreducible components of the source curve is preserved under automorphism of the dual complex. We then show that for any $p \geq 1$, a $p$-cell of $\T_{d, n}$ can be reconstructed from its ordered list of $(p-1)$-dimensional facets, together with this set of degrees. This reconstruction result then implies that $\Aut(\T_{d, n})$ injects into the set of permutations of its $0$-skeleton. A combinatorial analysis of the $1$-skeleton then allows us to conclude the result. For the $d = 1$ case, we first give an isomorphism
\begin{equation}\label{eqn:deg1}
    \T_{1, n} \cong \mathrm{Cone}(\T_{0, n+1}).
\end{equation}
The complex $\T_{0, n+1}$ is the dual complex of the boundary divisor in the Deligne--Mumford--Knudsen moduli space $\Mbar_{0, n+1}$. By proving that any automorphism of $\T_{1, n}$ preserves the cone vertex, we conclude that $\Aut(\T_{1, n}) \cong \Aut(\T_{0, n+1})$. The latter group was proven to be isomorphic to $\bbS_{n + 1}$ by Abreu--Pacini \cite{AbreuPacini}. Massarenti \cite{MassarentiFM} studied the automorphism group of $X[n]$; following his work, the identification  $\Aut(\T_{1, n}) \cong \bbS_{n + 1}$ implies that in general, automorphisms of the dual complex of the Fulton--MacPherson boundary do not extend to automorphisms of the whole space. This stands in contrast to $\Mbar_{g, n}$; for almost all $g$ and $n$, the automorphism groups of both $\Mbar_{g, n}$ and the dual complex of the boundary are given by $\bbS_n$ \cite{Massarenti, autdeltagn}.

Another consequence of (\ref{eqn:deg1}) is that the dual complex of any normal-crossings compactification of the configuration space $\Conf_n(X)$ is contractible for $n \geq 2$. This in particular implies that the weight-zero compactly supported-cohomology of $\Conf_n(X)$ vanishes for smooth, proper, and connected $X$; we expect that this fact is well-known to experts, but we include it as Corollary \ref{cor:wt0vanishing} for completeness.

\subsection{Related work} The dual complex $\T_{d, n}$ is closely related to familiar simplicial complexes in combinatorics, geometry, and topology:
\begin{itemize}
    \item When $d = 0$, the complex $\T_{0, n+1}$ is the link of the cone complex of phylogenetic trees, studied in e.g. \cite{BHV, RobinsonWhitehouse}, and is also known as the moduli space of $(n+1)$-pointed tropical curves of genus zero and volume one \cite{cgp}. The automorphism group of this complex has previously been computed by Abreu--Pacini \cite{AbreuPacini}, by Grindstaff \cite{Grindstaff}, and by the second-named author \cite{autdeltagn}.
    \item For $d \geq 0$, the complex $\T_{d, n}$ is obtained from the complex of separating curves on a genus-$d$ surface with $n$ boundary components by taking the quotient with respect to the action of the pure mapping class group $\mathrm{PMod}_{d, n}$. Kida showed the automorphism group of the complex of separating curves is given by the extended mapping class group $\mathrm{Mod}^{\pm}_{d, n}$ \cite{Kida}, in line with analogous results for the full curve complex \cite{Ivanov, Luo}. 
\end{itemize}
The automorphism group of the dual complex of the boundary divisor in $\Mbar_{g, n}$ was computed to be $\bbS_n$ for $g > 0$ in \cite{autdeltagn}. The automorphism group of the dual complex of the normal-crossings divisor of singular curves in  Hassett's compactification $\Mbar_{g, w}$ of $\M_{g, n}$ has been computed for $g > 0$ \cite{AutTropicalHassett}. The complex $\T_{0, n}$ was shown to be a \textit{flag} simplicial complex in \cite{GiansiracusaFlag}, and there has been interest in this property for higher-genus generalizations \cite{CLQFlag}. Our identification with $\T_{1, n}$ with the cone over $\T_{0, n+1}$ shows that $\T_{1, n}$ is also a flag complex. 

In contrast with several of the related complexes discussed above, the space $\T_{d, n}$ is contractible for $d \geq 1$, as was shown in \cite{vzdualcomplex}. On the other hand, $\T_{0, n}$ is homotopy equivalent to a wedge of $(n-2)!$ spheres of dimension $n-4$ \cite{RobinsonWhitehouse}, and the dual complex of the boundary of $\Mbar_{g, n}$ for $g > 0$ famously has extremely rich topology \cite{cgp}.

\subsection*{Outline} In \S\ref{sec:definitions}, we recall the definition of $\T_{d, n}$ as a symmetric $\delta$-complex. We then describe our strategy for computing $\Aut(\T_{d, n})$ in \S\ref{sec:strategy}, and also carry out the first step, which is to prove that degrees of stable maps along irreducible components are preserved by automorphisms. In \S\ref{sec:reconstruction}, we give the reconstruction algorithm which reduces the calculation of $\Aut(\T_{d, n})$ to its action on vertices, and then we compute $\Aut(\T_{d, n})$ for all $d \geq 2$ and $n \geq 0$ in \S\ref{sec:two_vertices}. Finally, we address the $d = 1$ case and discuss the relationship to the Fulton--Macpherson compactification in \S \ref{sec:degree_one}.

\subsection*{Acknowledgment}
SK is supported by NSF DMS-2401850.

\section{The boundary complex}\label{sec:definitions}
 We briefly recall the definition of $\T_{d, n}$. In \cite{vzdualcomplex}, $\T_{d, n}$ was proven to be the dual complex of the boundary divisor of $\Mbar_{0, n}(\P^r, d)$.\footnote{In that paper, the notation $\Delta_{0, n}(d)$ was used for $\T_{d, n}$.} The space $\T_{d, n}$ can be viewed as a moduli space of metric graphs: specifically, it is the moduli space of metric \textit{$(d, n)$-trees}, as defined below.

 \begin{defn}
    For integers $d, n \geq 0$, an \textit{$(d, n)$-tree} $\TT$ is a tuple 
    \[ \TT = (T, \delta_{\TT}, m_{\TT}) \]
    where:
    \begin{itemize}
    \item $T$ is a finite connected tree.
    \item $\delta_{\TT}: V(T) \to \Z_{\geq 0}$ is a function called the \textit{weight function}, and is required to satisfy
    \[ \sum_{v \in V(T)} \delta_{\TT}(v) = d. \]
    \item $m_{\TT} :\{1, \ldots, n\} \to V(T)$ is a function called the \textit{marking function}.
    \end{itemize}
    An $(d, n)$-tree is \textit{stable} if for all vertices $v \in V(T)$ with $\delta(v) = 0$, the inequality
    \[ \mathrm{val}(v) + |m^{-1}_{\TT}(v)| \geq 3 \]
    holds, where $\mathrm{val}(v)$ is the valence of the vertex $v$. 
 \end{defn}

The collection $\Tree_{d, n}$ of stable $(d, n)$-trees can be made into a category in the following way. The morphisms in the category $\Tree_{d, n}$ are defined to be compositions of isomorphisms (which are required to respect the weight and marking functions), and \textit{edge-contractions}, which we now define. We will generally write $E(\TT)$ and $V(\TT)$ for the edge set and vertex set, respectively, of the underlying tree.

\begin{defn}
Let $\TT$ be a stable $(d, n)$-tree, and suppose $e \in E(\TT)$ is an edge of $\TT$, connecting vertices $v_1, v_2 \in V(\TT)$. The \textit{edge-contraction of $\TT$ along $e$}, denoted by $\TT/e$, is the stable $(d, n)$-tree obtained by deleting $e$ from $\TT$ and identifying $v_1$ and $v_2$ into a single vertex $v$, which satisfies
\[ m_{\TT/e}^{-1}(v) = m_{\TT}^{-1}(v_1) \cup m_{\TT}^{-1}(v_2) \]
and
\[\delta_{\TT/e}(v) = \delta_{\TT}(v_1) + \delta_{\TT}(v_2). \]
% We write $c_e: \TT \to \TT/e$ to denote the morphism in the category $\Tree_{d, n}$ corresponding to the edge-contraction along $e$.
\end{defn}
If $\TT \to \TT/e$ is an edge-contraction of $\TT$, then there is a canonical injection
\[ E(\TT/e) \to E(\TT) \]
sending an edge in $\TT/e$ to its preimage in $E(\TT)$. 

\begin{exmp}\label{exmp:2,0}
    One sees that if $d = 2$ and $n = 0$, there are exactly two $(2, 0)$-trees: a single vertex with weight $2$, or a single edge with each vertex of weight $1$.
\end{exmp}
 
The $(d, n)$-trees are precisely the dual trees of genus zero stable maps to $\P^r$. Slightly informally, the space $\T_{d, n}$ can be viewed as the moduli space of metric stable $(d, n)$-trees, where a metric on a stable $(d, n)$-tree $\TT$ is a length function $\ell: E(\TT) \to \R_{\geq 0}$ of total length $1$. 

 A more precise description of $\T_{d, n}$ is given by its definition as a \textit{symmetric $\Delta$-complex} in the sense of \cite{cgp}. Then the \textit{geometric realization functor} of \cite[\S 3]{cgp} will realize $\T_{d, n}$ as a moduli space of metric graphs.
 
 Let us briefly recall the definition of symmetric $\Delta$-complexes. For an integer $p \geq 0$, write
 \[ [p] := \{0, \ldots, p\}. \]

 \begin{defn}
     A symmetric $\Delta$-complex $X$ is a functor
 \[X: I^{\mathrm{op}} \to \mathsf{Set}, \]
 where $I$ is the category whose objects are the finite sets $[p]$ for $p \geq 0$, and whose morphisms are injections. 
 \end{defn}
Symmetric $\Delta$-complexes have \textit{geometric realizations} as topological spaces with cell structures; we will not recall this construction here but instead refer the reader to \cite{cgp}. We now give the precise definition of $\T_{d, n}$ as a symmetric $\Delta$-complex.
\begin{defn}
We define the symmetric $\Delta$-complex
\[ \T_{d, n} : I^{\mathrm{op}} \to \mathsf{Set} \]
as follows: we set
\[ \T_{d, n}[p] = \{(\TT, \omega) \mid \TT \in \mathrm{Ob}(\Tree_{d, n}),\, \omega: E(\TT) \to [p] \text{ a bijection} \}/\sim \]
where $\sim$ is the equivalence relation generated by isomorphisms of edge-labelled pairs. Given an injection
\[ \iota: [q] \to [p], \]
we define
\[\T_{d, n}(\iota) : \T_{d, n}[p] \to \T_{d, n}[q] \]
by sequentially contracting a pair $(\TT, \omega) \in \T_{d, n}[p]$ along the edges $\omega^{-1}(i)$ for $i \notin \iota([q])$, and then taking the unique edge-ordering of the resulting tree which is compatible with $\omega$. 
\end{defn}

For ease of notation, given an injection $\iota: [q] \to [p]$, we will write $\iota^*$ for the set map
\[ \T_{d, n}(\iota) : \T_{d, n}[p] \to \T_{d, n}[q]. \]

Given $j \in [p]$, we define $\iota_j: [p - 1] \to [p]$ to be the unique order-preserving injection which misses $j$. For $(\TT, \omega) \in \T_{d, n}[p]$, we set the notation
\[ (\TT_j, \omega_j):= \iota_j^*(\TT, \omega). \]
We will compute the automorphism group of $\T_{d, n}$ in the category of symmetric $\Delta$-complexes. To be precise, a morphism of symmetric $\Delta$-complexes is a natural transformation of functors. We spell out what this means for automorphisms of $\T_{d, n}$. Note that if we set
\[\mathfrak{S}_p := \Aut([p]) \cong \bbS_{p + 1}, \]
then each $\kappa \in \mathfrak{S}_p$ defines a permutation
\[ \kappa^*: \T_{d, n}[p] \to \T_{d, n}[p]. \]
\begin{defn}
An \textit{automorphism} $\varphi$ of $\T_{d, n}$ is a collection of permutations
\[ \varphi_p: \T_{d, n}[p] \to \T_{d, n}[p] \]
for each $p \geq 0$, such that
\begin{enumerate}
    \item for all $\kappa \in \mathfrak{S}_{p}$, the diagram
    \[ \begin{tikzcd}
&\T_{d, n}[p] \arrow[d, "\kappa^*"] \arrow[r, "\varphi_p"] &\T_{d, n}[p] \arrow[d, "\kappa^*"]\\
&\T_{d, n}[p] \arrow[r, "\varphi_{p}"] &\T_{d, n}[p]
\end{tikzcd}\]
commutes, and
\item for each $p \geq 1$ and $j \in [p]$, the diagram
\[ \begin{tikzcd}
&\T_{d, n}[p] \arrow[d, "\iota_j^*"] \arrow[r, "\varphi_p"] &\T_{d, n}[p] \arrow[d, "\iota_j^*"]\\
&\T_{d, n}[p-1] \arrow[r, "\varphi_{p - 1}"] &\T_{d, n}[p-1]
\end{tikzcd}\]
commutes.
\end{enumerate}
\end{defn}

Given $\varphi \in \Aut(\T_{d, n})$, and $\TT \in \Tree_{d, n}$, we write $\varphi \TT$ for the underlying tree of the edge-labelled tree $\varphi(\TT, \omega)$ for any edge-labelling of $\TT$. One checks easily that $\varphi \TT$ is well-defined. Note that $\varphi$ also induces a biection $E(\TT) \to E(\varphi\TT)$, and there is a well-defined edge $\varphi e \in E(\varphi\TT)$ for each $e \in E(\TT)$. It is worth noting that automorphisms of $\T_{d, n}$ are exactly natural isomorphisms of functors $\T_{d, n} \to \T_{d, n}$.
\begin{exmp}\label{exmp:aut2,0}
    Unwinding the definition, one sees that $\Aut(\T_{2, 0})$ is trivial by Example \ref{exmp:2,0}.
\end{exmp}
\section{A strategy for computing $\Aut(\T_{d, n})$}\label{sec:strategy}
If we set
\[\T^{(p)}_{d, n} \subseteq \T_{d, n} \]
for the $p$-skeleton of $\T_{d, n}$ (i.e., the symmetric $\Delta$-complex obtained by taking the same $q$-cells of $\T_{d, n}$ for $q \leq p$, but which does not have any $q$-cells for $q > p$), then we get a restriction homomorphism
\begin{equation}\label{eqn:restrict}
    \Aut(\T_{d, n}) \to \Aut(\T_{d, n}^{(p)}) 
\end{equation}
for all $p \geq 0$. In order to prove that $\Aut(\T_{d, n}) = \bbS_n$, we will first prove that (\ref{eqn:restrict}) is injective for each $p \geq 0$. This fact would be clear if $\T_{d, n}$ were an ordinary simplicial complex, but it is not automatic for symmetric $\Delta$-complexes. We will then prove that when $p = 0$, the image of (\ref{eqn:restrict}) is equal to $\bbS_{n}$.
\subsection{The contraction deck}
For a $p$-simplex $(\TT, \omega) \in \T_{d, n}[p]$, we define the \textit{contraction deck} $\calD_{\TT, \omega}$ to be the set of edge-labelled pairs $(\TT_j, \omega_j^\star)$, where $\TT_j = \TT/\omega^{-1}(j)$, and
\[ \omega_j^\star: E(\TT_j) \to [p]\smallsetminus \{j\}  \]
is the bijection induced by
\[ E(\TT_j) \to E(\TT) \xrightarrow{\omega}[p]. \]
Thus $\calD_{\TT, \omega}$ is exactly the data of the ordered list of facets of $(\TT, \omega)$, with a slightly altered edge-labelling convention which is easier to work with: in particular, the $j$th facet is the unique element of $\calD_{\TT, \omega}$ whose set of edge-labels does not contain $j \in [p]$.

For any $\varphi \in \Aut(\T_{d, n})$ we  have
\[\varphi(\calD_{\TT, \omega}) = \calD_{\varphi(\TT, \omega)}\]
by definition. Therefore, if every $p$-simplex in $\T_{d, n}$ were uniquely determined by $\calD_{\TT, \omega}$, then the restriction map
\[ \T^{(p)}_{d, n}  \to \T_{d, n}^{(p-1)}\]
would be an injection. Unfortunately, this is not always true: there exist distinct $p$-simplices with the same ordered list of facets, as in Figure \ref{fig:contraction_deck_counterexmp}. To get around this, we proceed as follows: first, for $\TT \in \Tree_{d, n}$, we define $\calW_{\TT}$ to be the \textit{multiset} of node weights; that is $\calW_{\TT} \subseteq (\Z_{\geq 0})^2$ is the set of ordered pairs $(j, k)$ where $j$ is a vertex weight of $\TT$, and $k$ is the number of vertices of $\TT$ with weight $j$. We will show that 
\[ \calW_{\TT} = \calW_{\varphi\TT} \]
for any $\TT$. Finally, we will show that together, the invariants $\calD_{\TT, \omega}$ and $\calW_{\TT}$ determine $(\TT, \omega) \in \T_{d, n}[p]$ uniquely, as long as $p \geq 1$.
\subsection{Preservation of vertex weights}
To show that $\calW_{\TT}$ is preserved under automorphisms, we first find a simplex that is ``most symmetric" in $\T_{d, n}$, which must be fixed under automorphism. We define $\mathbf{C} \in \Tree_{d, n}$ to be a decorated \textit{star graph} with $d$ edges: this means the underlying tree of $\mathbf{C}$ consists of a central vertex $v$ connected by $d$ distinct edges to $d$ distinct leaf vertices. Then $\mathbf{C}$ is determined by requiring that each leaf vertex has weight $1$, and that all of the markings supported on the central vertex (here we assume that $d \geq 2$ and $n \geq 1$ if $d = 2$). Then all edge-orderings of $\mathbf{C}$ are equivalent, since $\Aut(\mathbf{C})$ acts strongly transitively on $E(\mathbf{C})$. It is also straightforward to see that for $d \geq 2$, $\mathbf{C}$ is the unique tree with $d$ edges in $\Tree_{d, n}$ which has this property. Therefore, picking arbitrarily an edge-ordering $\omega$, we have the following lemma. 
\begin{lem}
Suppose $d \geq 2$ and $n \geq 0$, with $n \geq 1$ if $d = 2$. Then for any $\varphi \in \Aut(\T_{d, n})$, we have
\[ \varphi(\mathbf{C}, \omega) = (\mathbf{C}, \omega). \]
\end{lem}
Since $(\mathbf{C}, \omega)$ is preserved under automorphism, the same is true of any of its faces. In particular, let $\B_{1, \varnothing} \in \T_{d, n}[0]$ be the unique vertex of $\T_{d, n}$ which is contained in the cell $(\mathbf{C},\omega)$. The underlying tree of $\B_{1, \varnothing}$ is a single edge connecting two vertices, one of which is of weight one and does not support any markings, so that the other vertex is of weight $d - 1$, and supports all of the markings in $\{1, \ldots, n\}$. Then $\varphi \B_{1, \varnothing} = \B_{1, \varnothing}$. We will conflate $\B_{1, \varnothing}$ and its underlying tree, since it has a unique edge-labelling. Since $\B_{1, \varnothing}$ is the unique vertex of $(\mathbf{C}, \omega)$, we immediately derive the following lemma.
\begin{lem}
For any $\varphi \in \Aut(\T_{d, n})$, we have
\[\varphi \B_{1, \varnothing} = \B_{1, \varnothing}.  \]
\end{lem}

\begin{defn}
Let $\TT \in \Tree_{d, n}$. An edge $e \in E(\TT)$ is called a \textit{$1$-end} if it separates an unmarked leaf of weight one from the rest of $\TT$.
\end{defn}

An edge $e \in E(\TT)$ is a $1$-end if and only if the graph obtained from $\TT$ by contracting all edges except for $e$ is equal to $\B_{1, \varnothing}$. Since $\B_{1, \varnothing}$ is preserved by $\Aut(\T_{d, n})$, we know that $1$-ends are preserved:
\begin{lem}
Let $\TT \in \Tree_{d, n}$, and suppose $e \in E(\TT)$ is a $1$-end of $\TT$. Then $\varphi e$ is a $1$-end of $\varphi \TT$. 
\end{lem}
The key result which allows us to deduce that the multiset of weights of any $\TT \in \Tree_{d, n}$ is preserved is given by the following proposition.
\begin{prop}\label{prop:1endclups}
    Let $\TT \in \Tree_{d, n}$. Then $e_0$ and $e_1$ are $1$-ends of $\TT$ which share a vertex if and only if $\varphi e_0$ and $\varphi e_1 \in E(\TT)$ are also $1$-ends sharing a vertex.
\end{prop}
\begin{proof}
Let $\tau$ be an edge-labelling of $\TT$, and define
\[ (\varphi \TT, \varphi \tau): = \varphi (\TT, \tau), \]
so $\varphi \tau$ is some edge-labelling of $\varphi \TT$ witnessing the image of the cell $(\TT, \tau)$ under $\varphi$. We can suppose without loss of generality that $\tau^{-1}(0) = e_0$ and $\tau^{-1}(1) = e_1$. By the preceding lemma we must have that $\varphi e_0$ and $\varphi e_1$ are $1$-ends of $\varphi \TT$. Moreover, since $e_0$ and $e_1$ can be exchanged under an automorphism of $\TT$ that fixes every other edge, $\varphi e_0$ and $\varphi e_1$ must also have this property. When $|E(\TT)| \geq 4$, this immediately proves that $\varphi e_0$ and $\varphi e_1$ must be $1$-ends supported at the same vertex; the proposition in this case is proved by applying the same argument to $\varphi^{-1}$. 

The only remaining case to to consider is $|E(\TT)| = 3$, and $\TT$ is a path with three edges, such that the two extreme edges are $e_0$ and $e_1$, and $\TT$ has an automorphism which exchanges $e_0$ and $e_1$ and preserves the middle edge, since this is the only way we can have an automorphism exchanging the two $1$-ends and fixing every other edge, without the two $1$-ends beig supported on the same vertex. We need to rule out that $\varphi \TT$ is a star with $3$ edges, supporting $1$-ends $\varphi e_0$ and $\varphi e_1$. 

Suppose for contradiction that $\varphi$ has this property. For $\TT$ to have such an automorphism, it must be that $n = 0$, and if we embed $\TT$ in the plane, then the vertex weights of $\TT$ read $(1, e, e, 1)$ from right to left, where $2e + 2 = d$. Thus $\varphi \TT$ is a star with three edges, such that two leaves have weight $1$, one leaf has weight $e$, and the central vertex has weight $e$. Also, we must have $e > 1$ since otherwise $\varphi\TT$ would have more $1$-ends than $\TT$. Now consider the expansion $\TT^{\mathrm{sp}}$ of $\TT$ obtained by connecting $e$ many $1$-ends to each of the two central vertices, which are then reassigned weight $0$. Then $\TT^{\mathrm{sp}}$ consists of exactly $1$ edge which is not a $1$-end, such that each vertex of this edge supports $e + 1$ many $1$-ends. Since $|E(\TT^{\mathrm{sp}})| \geq 4$, the first part of the proposition implies that $\varphi\TT^{\mathrm{sp}} = \TT^{\mathrm{sp}}$. Now one verifies that $\TT^{\mathrm{sp}}$ has no contraction which is equal to $\varphi\TT$, which is a contradiction. This completes the proof.
\end{proof}

\begin{cor}
For any $\TT \in \Tree_{d, n}$ and $\varphi \in \Aut(\T_{d, n})$, we have \[\calW_{\TT} = \calW_{\varphi \TT}.\]
\end{cor}
\begin{proof}
Define the \textit{sprouting} $\TT^{\mathrm{sp}}$ of $\TT$ by taking each vertex $v \in V(\TT)$ with either $\delta(v) > 1$ or $\delta_{\TT}(v) = 1$ and $\val(v) \geq 2$, and adding $\delta_{\TT}(v)$ edges with are $1$-ends, all of which share the vertex $v$, which is now of weight zero. Now choose an edge-labellings $\tau^{\mathrm{sp}}$ and $\tau$ of $\TT^{\mathrm{sp}}$ and $\TT$, respectively, so that $\tau$ is obtained from $\tau^{\mathrm{sp}}$ by contracting all of the $1$-ends which were added to $\TT$. Then, contracting the images of these $1$-ends in $\varphi(\TT^{\mathrm{sp}}, \tau^{\mathrm{sp}})$, we obtain $\varphi(\TT, \tau)$. The corollary now follows from Proposition \ref{prop:1endclups}.  
\end{proof}

\section{A reconstruction algorithm}\label{sec:reconstruction}
We now present a simple algorithm for the reconstruction of a cell $(\TT, \omega)$ of $\T_{d, n}$ from its contraction deck $\calD(\TT, \omega)$ together with the multiset of weights $\calW_{\TT}$, under the assumption that $E(\TT) \geq 2$. 
\subsection{Node types}
\begin{defn}
    The \textit{node type} of a vertex $v \in \TT$ is the tuple \[ \mathrm{type}(v):= (\delta_{\TT}(v), \mathrm{deg}(v) + |m^{-1}_{\TT}(v)|).\]
\end{defn}

Let $\prec$ be the lexicographic order on $\Z^{2}$. Suppose $(\TT, \omega) \in \T_{d, n}[p]$ with edge $e \in E(\TT)$ connecting vertices $v_1$ and $v_2$, and let $v_{\mathrm{new}}$ be the vertex replacing $v_1$ and $v_2$ in $\TT/e$. Then, by the stability condition, one checks that there is a strict inequality
\begin{equation}\label{eqn:ineq}
\mathrm{type}(v_i) \prec \mathrm{type}(v_{\mathrm{new}}) \end{equation}
for $i = 1, 2$. Therefore, the first step of the algorithm is to choose, of all of the vertices of trees in the contraction deck, a vertex $v_{\mathrm{max}}$ of a tree which has maximal type, under the lexicographic order. This choice may be non-canonical. Without loss of generality, suppose the chosen vertex is in the tree $\TT_0$. From (\ref{eqn:ineq}), we know that $(\TT, \omega)$ must be obtained from $(\TT_0, \omega_0^\star)$ by expanding $\TT_0$ at $v_{\mathrm{max}}$ and labelling the new edge by $0$. Let $v_1$ and $v_2$ be the preimages of $v_{\mathrm{max}}$ in $\TT$. From $\calW_{\TT}$ we know the unique integers $x, y$ such that $\{\delta_{\TT}(v_1), \delta_{\TT}(v_2) \} = \{x, y\}$. We suppose without loss of generality $\delta_{\TT}(v_1) = x$ and $\delta_{\TT}(v_2) = y$.

\subsection{Partitioning edges and markings} Now let $M = m_{\TT_0}^{-1}(v_{\max})$, and set \[S = \omega_0^\star(\{ e \in E(\TT_0) \mid v_{\max}  \in e \}) \subseteq [p].\] Then $S$ is exactly the set of labellings of edges in $\TT$ which are incident to edge $0$. The cell $(\TT, \omega)$ is uniquely determined by a pair of surjective maps
\[ \alpha: S\to \{v_1, v_2\}, \]
and
\[ \beta: M \to \{v_1, v_2\}, \]
subject to some stability conditions. The map $\alpha$ determines how to split the edges at $v_{\max}$, and the map $\beta$ determines how to split the markings. The problem is now to reconstruct $\alpha$ and $\beta$ from $\calD(\TT, \omega)$ and $\calW_{\TT}$.

For $j \in [p]\smallsetminus \{0\}$, let us denote by  \[S(\TT_j, \omega_j^\star) \subseteq [p] \smallsetminus \{j\}\]
the set of labellings of edges which are incident to edge $0$ in $\TT_j$,
and
\[M(\TT_j, \omega_j)  \subseteq \{1, \ldots, n\}\]
the set of markings which are incident to edge $0$ in $\TT_j$.
\subsubsection{Case I} If we can find $j$ such that
\[S(\TT_j, \omega_j^\star) = S, \]
then the set of edges and marks incident to the vertices of edge $0$ are $\alpha$ and $\beta$ respectively, as edge $j$ in $\TT$ was not incident to either vertex of edge $0$ in $\TT$.
\subsubsection{Case II} Now suppose that for all $j \in [p]\smallsetminus\{0\}$, we have
\[ S(\TT_j, \omega_j^\star) \neq S. \]
This implies that all edges $j \in [p]\smallsetminus\{0\}$ in the original tree $T$ are incident to edge $0,$ so $S = [p] \smallsetminus \{0\}$. Then $\TT_j$ is obtained from $\TT$ by contracting edge $j$ and replacing its incident leaf vertex $v_j$ and some $v$ in edge $0$ with $v_{new}$. The two vertices of edge $0$ in $\TT_j$ can be used to uniquely determine $\delta_{\TT}(v_j).$ Then the markings $m_{\TT}^{-1}(v_j)$ are determined by taking the markings incident to edge $0$ in $\TT_j$ that are not elements of $M.$ 

Finally, using the fact that at least one of $w(v_j) > 0, |m_{\TT}^{-1}(v_j)| > 0$ must be true for $v_j$, we can identify $v_{new}$ in $\TT_j$ among the vertices of edge $0.$ This determines $\TT$ uniquely.

By the reconstruction algorithm, we have now proved the following theorem.
\begin{thm}\label{thm:restriction_inj}
    For $d \geq 2$ and $n \geq 0$, the natural restriction homomorphism
    \[ \Aut(\T_{d, n}) \to \Aut(\T_{d, n}^{(0)}) \]
    to permutations of the $0$-skeleton is an injection.
\end{thm}

Since $\calW_{\TT}$ is preserved by $\Aut(\T_{d, n})$ for any $\TT$, we immediately get the $n = 0$ case of our main theorem.

\begin{cor}\label{cor:unmarked_case}
    For $d \geq 2$, the automorphism group $\Aut(\T_{d, 0})$ is trivial.
\end{cor}

\section{Trees with two vertices}\label{sec:two_vertices}
In view of Theorem \ref{thm:restriction_inj}, we now study the action of $\varphi \in \Aut(\T_{d,n})$ on those trees $\G \in \T_{d,n}$ with $|V(\G)| = 2$, when $d \geq 2$. By Corollary \ref{cor:unmarked_case}, we may as well assume that $n > 0$. Let us set some notation for these trees.
\begin{defn}\label{defn:two_vertex_tree}
Suppose given an integer $e \leq d$ and a subset $S \subseteq \{1, \ldots, n\}$; if $e = 0$, we require that $|S| \geq 2$, and if $e = d$, we require that $|S| \leq n-2$. We define
\[ \B_{e, S} \in \T_{d,n}\]
to be the tree with two vertices $v$ and $w$, such that $\delta_{\B_{e, S}}(v) = e$ and $m^{-1}_{\B_{e, S}}(v) = S$. This condition implies that $\delta_{\B_{e, S}}(w) = d-e$, and $m^{-1}_{\B_{e, S}}(w) = S^c$.
\end{defn}

Given $\varphi \in \Aut(\T_{d, n})$, we want to show that there exists $\sigma \in \bbS_n$ such that $\varphi \B_{e, S} = \sigma \B_{e, S}$ for all $\B_{e, S}$. The following lemma is a straightforward calculation. 

\begin{lem}\label{lem:counting}
Let $u_{e, S}$ denote the number of $1$-cells in $\T_{d,n}$ which contain $\B_{e, S}$ as a vertex. Then
\[ u_{e, S} = (e + 1)\cdot 2^{|S|} + (d- e + 1) \cdot 2^{|S^c|} - n - 4. \]
\end{lem}

\begin{cor}
For all $d \geq 2$, we have $\Aut(\T_{d, 1}) \cong \bbS_1$.
\end{cor}
\begin{proof}
Suppose that
\[ \varphi \B_{e, \{1\}} = \B_{e', \{1\}}. \] Lemma \ref{lem:counting} implies
\[ 2(e + 1) + (d - e + 1) = 2(e' + 1) + (d - e' + 1), \]
so $e = e'$, and we conclude that $\Aut(\T_{d, 1}) \cong \bbS_1$ for any $d \geq 2$ (and it is worth noting that $\T_{1, 1}$ is an empty complex). 
\end{proof}

\begin{cor}\label{cor:wt0orbits}
Suppose $d, n \geq 2$ and let $\varphi \in \Aut(\T_{d, n})$. Let $S \subseteq \{1, \ldots, n\}$. Then
\[ \varphi \B_{0, S} = \B_{0, R} \]
for some $R$ with $|S| = |R|$.
\end{cor}
\begin{proof}
Since $\varphi$ preserves the set of vertex weights we have $\varphi \B_{0, S} = \B_{0, R}$ for some $R \subseteq \{1, \ldots, n\}$. Suppose towards a contradiction that $|S| = k$ and $|R| = j$, with $k \neq j$. Without loss of generality assume $j >k$. Then by Lemma \ref{lem:counting} we must have
\[ 2^{k} + (d + 1)2^{n - k} = 2^j + (d+1) 2^{n-j},\]
which implies that
\[d + 1 = \frac{2^j - 2^k}{2^{n- k} - 2^{n -j}} = \frac{2^k(2^{j - k }-1)}{2^{n - j}(2^{j - k } - 1)} = \frac{2^k}{2^{n - j}}. \]
But $k + j = n$, so the above implies $d = 0$, which is a contradiction.
\end{proof}

\begin{lem}\label{lem:Kneser}
Suppose $d \geq 2$ and $n \neq 1,2$. Then there exists a unique $\sigma \in \bbS_n$ such that
\[ \varphi \B_{0, \{i,j\}} = \sigma \B_{0, \{i,j\}} \]
for each two-element subset $\{i, j\} \subset \{1, \ldots, n\}$.
\end{lem}

\begin{proof}
From Corollary \ref{cor:wt0orbits} we have
\[ \varphi \B_{0, \{i, j\}} = \B_{0, \{k, \ell\}} \]
for some two-element subset $\{k, \ell\}$. Now observe that two vertices of the form $\B_{0, \{i, j\}}$ and $\B_{0, \{k, \ell\}}$ are connected by a $1$-cell in $\T_{d, n}$ if and only if $\{i, j\}$ and $\{k, \ell\}$ are disjoint. This means that any automorphism $\varphi$ of $\T_{d, n}$ must in particular induce a permutation of the two-element subsets $\{i, j\}$ of $\{1, \ldots, n\}$ which preserves disjointness. Such permutations are precisely automorphisms of the \textit{Kneser graph} $K(n, 2)$. The automorphism group of $K(n, 2)$ is known to be $\bbS_n$ when $n \geq 5$ and when $n = 3$ \cite{BraunKneser}; this proves the lemma in these cases.

Now suppose $n = 4$. Notice that $\B_{0, \{k, \ell\}}$ and $\B_{0, \{i, j\}}$ are of distance $2$ in the $1$-skeleton of $\T_{d, 4}$ (meaning the minimal path length between them is $2$) if and only if $\{k, \ell\}$ and $\{i, j\}$ share exactly one element. This means that the action of any $\varphi \in \Aut(\T_{d, 4})$ permutes the elements $\B_{0, \{i, j\}}$ in such a way that the cardinality of the intersection of two $2$-element subsets is preserved. The group of such permutations is checked to be isomorphic to $\bbS_4 \rtimes \bbS_2$, generated by the natural $\bbS_4$ action and the complementation map $\{i, j\} \mapsto \{i, j\}^c$. Thus we only need to rule out that $\varphi$ is acting by the complementation map on the $\B_{0, \{i, j\}}$, so suppose toward a contradiction that this is the case. Consider the vertex $\B_{0, \{2, 3, 4\}}$. By Corollary \ref{cor:wt0orbits}, we have \[\varphi \B_{0, \{2, 3, 4\}} = \B_{0, \{i, j, k\}}\] for some $3$-element subset $\{i, j, k\} \subseteq \{1, 2, 3, 4\}$. But among the $\B_{0, \{k, \ell\}}$, the vertex $\B_{0, \{2,3,4\}}$ is adjacent to exactly $\B_{0, \{2, 3\}}$, $\B_{0, \{2, 4\}}$, and $\B_{0, \{3, 4\}}$. Under the complementation map, this set of vertices is mapped to $\B_{0, \{1, 4\}}$, $\B_{0, \{1, 3\}}$, and $\B_{0, \{1, 2\}}$, but there is no $\B_{0, \{i, j, k\}}$ that is adjacent to exactly these three $\B_{0, \{k, \ell\}}$, which is a contradiction.
\end{proof}

\begin{thm}
    Suppose $d \geq 2$ and $n \neq 1, 2$. Then there exists a unique $\sigma \in \bbS_n$ such that $\varphi \B_{e,S} = \sigma \B_{e, S}$ for all $e$ and $S$. In particular, $\Aut(\T_{d,n}) \cong \bbS_n$ for $d \geq 2$ and $n \geq 4$.
\end{thm}
\begin{proof}
A vertex $\B_{e, S} \neq \B_{0, \{i, j\}}$ is connected to the vertex $\B_{0, \{i, j\}}$ by a $1$-cell in $\T_{d, n}$ if and only if $\{i, j\} \subseteq S$ or $\{i, j\} \subseteq S^c$. Therefore, we must have either $\varphi \B_{e, S} = \B_{e, \sigma(S)}$ or $\varphi\B_{e,S}  = \B_{e, \sigma(S^c)}$, where $\sigma \in \bbS_n$ is given by Lemma \ref{lem:Kneser}. If the first case always holds, then we are done. We can study the second case using Lemma \ref{lem:counting}: if $\varphi\B_{e,S}  = \B_{e, \sigma(S^c)}$, then we must have $u_{e, S} = u_{e, S^c}$, so
\[ (e + 1) \cdot 2^{|S|} + (d - e + 1) \cdot 2^{|S^c|} = (e + 1) \cdot 2^{|S^c|} + (d - e + 1) \cdot 2^{|S|}, \]
or equivalently
\[(e + 1) \cdot (2^{|S|} - 2^{|S^c|}) = (d - e + 1)\cdot(2^{|S|} - 2^{|S^c|}).  \]
This equality holds if and only if $|S| = |S^c|$ or $e = d/2$. If $e = d/2$, then $\B_{e, \sigma(S^c)} = \B_{e, \sigma(S)}$, so we need only rule out the case where $e \neq d/2$, $|S| = |S^c|$, and $\varphi\B_{e, S} = \B_{e, \sigma(S^c)}$. Suppose towards a contradiction that this holds. Set $d_1 = e$ and $d_2 = d - e$. At least one of $S$ or $S^c$ is of cardinality at least $2$; suppose without loss of generality that $|S| \geq 2$. Choose some $\{i, j\} \subseteq S$, and consider the $1$-cell $(\TT, \omega)$ of $\T_{d, n}$ given in Figure \ref{fig:1cella}. 
\begin{figure}
    \centering
    \includegraphics{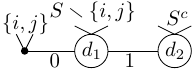}
    \label{fig:1cella}
\end{figure}
Since $(\TT_1, \omega_1) = \B_{0, \{i, j\}}$, we must have that $\varphi(\TT_1, \omega_1) = \B_{0, \{k, \ell\}}$, where \[\{k, \ell\} = \sigma(\{i, j\}). \] But $(\TT_0, \omega_0) = \B_{e, S}$, so $\varphi(\TT_0, \omega_0) = \B_{e, \sigma(S^c)}$. Thus $\varphi(\TT, \omega)$ must be as in Figure \ref{fig:1cellb}, and we must have that $A \cup \{k, \ell\} = \sigma(S^c)$ and $B = \sigma(S)$. In particular, $\{k, \ell\} \subseteq \sigma(S^c)$, which is a contradiction.
\begin{figure}
    \centering
    \includegraphics{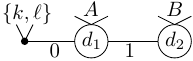}
    \label{fig:1cellb}
\end{figure}
\end{proof}

We conclude by outlining the proof of the $n = 2$ case; the interested reader may fill in the details.

\begin{prop}
For all $d \geq 2$, we have
\[ \Aut(\T_{d, 2}) \cong \bbS_2. \]
\end{prop}
\begin{proof}
When $d = 2$, one can prove this by direct analysis of the top-dimensional cells of the complex, so we suppose $d \geq 3$.

First use Lemma \ref{lem:counting} to prove that $\bbS_2$-orbits of vertices are preserved. We would like to show that if
\[ \varphi \B_{e, \{1\}} = \varphi \B_{e, \{2\}} \] for some $e \neq d/2$, then $\varphi$ acts as the transposition on all $\B_{e', \{1\}}$ for $e' \neq d/2$. To see this, one may suppose for sake of contradiction that $\varphi \B_{e', \{1\}} = \B_{e', \{1\}}$, where $e' \neq d/2$ and without loss of generality $e' < e$, and then consider the action of $\varphi$ on a path with $3$ vertices, such that the leftmost vertex has weight $e'$ and supports marking $1$, and the rightmost vertex has weight $d - e$ and supports marking $2$, so that the middle vertex supports no markings and has weight $e - e'$; one arrives at a contradiction by considering possible images of this path.
 \end{proof}

\section{The case $d = 1$ and the Fulton--MacPherson compactification}\label{sec:degree_one}
When $d = 1$ and $n \geq 2$, there is a canonical isomorphism of simple normal-crossings pairs
\[ \Mbar_{0, n}(\P^1, 1) \cong \P^1[n],\]
where $\P^1[n]$ is the Fulton--MacPherson compactification of the ordered configuration space $\Conf_n(\P^1)$, and the snc divisors are the boundary divisors in each case. Thus the dual complex of
\[ \Conf_n(\P^1) \subset \P^1[n] \]
is canonically identified with $\TT_{1, n}$ for $n \geq 2$. This extends to arbitrary smooth proper varieties $X$: the dual complex of
\[ \Conf_n(X) \subset X[n] \]
does not depend on $X$, as is clear from Fulton--MacPherson's original blow-up construction of $X[n]$. Therefore the dual boundary complex of $X[n]$ is identified with $\T_{1, n}$ for any smooth proper variety $X$ and any integer $n \geq 2$. In this section we will compare the dual complex $\T_{1, n}$ to the dual complex $\T_{0, n+1}$ of the boundary divisor of the moduli space $\Mbar_{0, n+1}$. 

\subsection{From abstract simplicial complexes to symmetric $\Delta$-complexes} We will use that the groupoids $\Tree_{1, n}$ and $\Tree_{0, n + 1}$ are automorphism-free in order to think of their dual complexes $\T_{1, n}$ and $\T_{0, n+1}$ as \textit{abstract simplicial complexes}. To be precise, an abstract simplicial complex is a pair where $V$ is a finite set of vertices (0-cells),  and $\calF \subseteq 2^V$ is a collection of subsets of $V$ which is closed under the operation of taking subsets. The set
\[K(p):= \{F \in \Delta \mid |F| = p+1 \} \subset \calF \]
is the set of $p$-simplices of $K$. There is a full and faithful functor from the category of abstract simplicial complexes to the category of symmetric $\Delta$-complexes, taking an abstract simplicial complex $K$ symmetric $\Delta$-complex $\Delta_K$ given by
\[ \Delta_K[p]= \{ (F, \omega) \mid F \in K(p),\, \omega : F \to [p] \text{ a bijection} \}, \]
with the natural face maps.

The dual complex $\T_{0, n + 1}$ can be described as an abstract simplicial complex: the vertices correspond to those graphs in $\Tree_{0, n + 1}$ with exactly one edge, and a collection of $p + 1$ vertices spans a $p$-simplex if and only if there exists a tree in $\T_{0, n + 1}$ which admits a contraction to each of the corresponding one-edge graphs. The complex $\T_{1, n}$ admits an analogous description. 

If $K = (V_K, \calF_K)$ is an abstract simplicial complex, we define a new abstract simplicial complex $\mathrm{Cone}(K)$, called the \textit{cone over $K$}, with vertex set $V_K \cup \{\star\}$ and face set
\[ \calF_K \cup\{ F \cup \{\star\} \mid F \in \calF_K\}. \]
In other words, we add a single new vertex $\star$ to $K$, and then for each $p$-dimensional face $F$ of $K$ we add a new $p+1$-dimensional face which contains $F$ as a facet and $\star$ as a vertex. Passing to the geometric realization of $K$, this operation corresponds to taking the topological cone over $K$.
\begin{thm}\label{thm:cone_iso}
For $n \geq 2$, there is a canonical isomorphism of abstract simplicial complexes
\[ \T_{1, n} \cong \mathrm{Cone}(\T_{0, n+1}). \]
\end{thm}
\begin{proof}
First observe that there is an embedding of abstract simplicial complexes
\[ \T_{0, n + 1} \hookrightarrow \T_{1, n} \]
determined by replacing the $(n + 1)$st marking on a tree $\TT \in \Tree_{0, n+1}$ with a vertex weight of weight $1$. In this way, we identify $\T_{0, n+1}$ with the subcomplex of $\T_{1, n}$ determined by those trees which do not contain any $1$-end.  Observe that every cell in $\T_{1, n} \smallsetminus \T_{0, n+1}$, other that $\B_{1, \varnothing}$, is obtained from a unique tree $\TT\in \Tree_{0, n+1}$ by sprouting the unique vertex of weight $1$ in order to create a $1$-end. The resulting cell always has $\B_{1, \varnothing}$ as a vertex. We have obtained $\T_{1, n}$ from $\T_{0, n+1}$ by first adding the vertex $\B_{1, \varnothing}$, and then adding exactly one new simplex for each $f$ in $\T_{0, n+1}$, which is spanned by $f$ and $\B_{1, \varnothing}$. This construction characterizes the cone over $\T_{0, n+1}$.
\end{proof}

\begin{rem}
    The cone $\mathrm{Cone}(K)$ over a simplicial complex $K$ is a special case of the \textit{join} $K *K'$ of two simplicial complexes $K$ and $K'$. This complex has $V_{K*K'} = V_K \cup V_{K'}$ and \[\calF_{K * K'}  =F \cup F' \mid F \in \calF_{K} \text{ and } F' \in \calF_{K'}\}. \] Then $\mathrm{Cone}(K)$ is the join of $K$ with a single vertex. Theorem \ref{thm:cone_iso} can be generalized to prove that
    \[ \T_{d, n} \cong (\T_{0, n +d} * \sigma_d)/\bbS_d, \]
    where $\sigma_d$ is the standard $d$-simplex, with vertex set $[d]$ and face set $2^{[d]}$. The quotient is taken in the category of symmetric $\Delta$-complexes.
 \end{rem}

\begin{cor}
For $n \geq 4$, we have $\Aut(\T_{1, n}) \cong \bbS_{n + 1}$. When $n = 3$, we have $\Aut(\T_{1, 3}) \cong \bbS_{3}$, while $\Aut(\T_{1, 2})$ is trivial. 
\end{cor}
\begin{proof}
First, observe that the vertex $\B_{1, \varnothing}$ is preserved by every automorphism of $\T_{1, n}$: when $n \geq 3$, it is seen to be the unique vertex of $\T_{1, n}$ which is connected by a $1$-cell to every other vertex of $\T_{1, n}$. When $n = 2$, we have that $\T_{1, 2}$ is the single vertex $\B_{1, \varnothing}$. Therefore, we have
\[ \Aut(\T_{1, n}) \cong \Aut(\T_{0, n + 1}) \]
for all $n \geq 3$. The theorem now follows from Abreu--Pacini's result that $\Aut(\T_{0, n+1}) \cong \bbS_{n + 1}$ for $n \geq 4$, and $\Aut(\T_{0, 4}) \cong \bbS_3$ \cite{AbreuPacini}. 
\end{proof}

Our calculation of $\Aut(\T_{1, n})$ suggests that usually, the map of pairs
\[ \Aut(X[n], \Conf_n(X))  \to \Aut(\T_{1, n})  \]
is not surjective (here we have set $\Aut(Y, Z)$ to be the group of automorphisms of $Y$ which preserve the subspace $Z$). For example, if $C$ is a curve of genus $\neq 1$, Masseranti has shown that $\Aut(C[n]) \cong \Aut(C) \times \bbS_n$ \cite{MassarentiFM}, and conjectures that $\Aut(X[n]) \cong \Aut(X) \times \bbS_n$ when $X$ is of general type. In particular, his result also implies that $\Aut(\Mbar_{0, n}(\P^1, 1)) \cong \Aut(\P^1) \times \bbS_n$, so $\T_{1, n}$ has automorphisms which do not extend to any automorphism of $\Mbar_{0, n}(\P^1, 1)$.

Since the cone over any simplical complex is contractible, we immediately have the following corollary, which is likely well-known to experts, and is already implicit in \cite{vzdualcomplex}, where $\T_{d, n}$ was shown to be contractible for all $d \geq 1$. 
\begin{cor}\label{cor:wt0vanishing}
Suppose $X$ is a smooth, proper, and connected variety over $\C$. Then we have the vanishing result
\[W_0H^\star_c(\Conf_n(X);\QQ) = 0 \]
for the weight-zero compactly-supported rational cohomology of $\Conf_n(X)$ when $n \geq 2$.
\end{cor}
\begin{proof}
    For any smooth variety $Y$, Deligne's weight spectral sequence identifies the weight-zero compactly-supported rational cohomology of $Y$ with the reduced rational cohomology of the dual complex of the boundary $\Delta(\overline{Y} \smallsetminus Y)$ in any normal-crossings compactification $Y \hookrightarrow \overline{Y}$; see \cite{cgp}. Now we apply this result when $Y = \Conf_n(X)$ and $\overline{Y} = X[n]$ is the Fulton--Macpherson compactification. The dual boundary complex is $\T_{1, n}$, which is a cone and therefore contractible.
\end{proof}
Finally, the identification of Theorem \ref{thm:cone_iso} implies that the dual complex $\T_{1, n}$ of $\Conf_n(X) \subset X[n]$ is a \textit{flag complex}: it is the maximal simplicial complex on its $1$-skeleton.
\begin{cor}
The simplicial complex $\T_{1, n}$ is a flag complex for all $n \geq 2$. 
\end{cor}
\begin{proof}
    Giansiracusa \cite{GiansiracusaFlag} proved that $\T_{0, n}$ is a flag complex. It is straightforward to see that the cone over a flag complex must also be a flag complex, so the result follows from Theorem \ref{thm:cone_iso}.
\end{proof}

\bibliographystyle{amsalpha}
\bibliography{reference}
% \begin{lem}
% For each $\B_{e, S} \in \T_{d,n}$ as in Definition \ref{defn:two_vertex_tree}, there exists some $\sigma \in S_n$ such that we have
% \[ \varphi \B_{e, S} = \sigma\B_{e,S}. \]
% \end{lem}
\end{document}